\documentclass[10pt,a4paper]{article}

\usepackage{amsmath,amssymb,amsthm,mathtools}
\usepackage{microtype}
\usepackage{enumitem}
\usepackage[hidelinks]{hyperref}
\usepackage[T1]{fontenc}
\usepackage{lmodern}

\setlist[enumerate]{label=(\roman*),leftmargin=2.3em,itemsep=2pt,topsep=4pt}

\newtheorem{theorem}{Theorem}[section]
\newtheorem{proposition}[theorem]{Proposition}
\newtheorem{lemma}[theorem]{Lemma}
\newtheorem{corollary}[theorem]{Corollary}
\theoremstyle{definition}
\newtheorem{definition}[theorem]{Definition}
\newtheorem{example}[theorem]{Example}
\theoremstyle{remark}
\newtheorem{remark}[theorem]{Remark}

\newcommand{\AS}{Amitsur--Small}
\newcommand{\cent}{\mathrm{C}}

\title{Characterization of the centrally finite Amitsur--Small division rings}
\author{Elad Paran\\[2pt]
\small Department of Mathematics and Computer Science\\
\small The Open University of Israel\\
\small \texttt{paran@openu.ac.il}}
\date{}

\hypersetup{
  pdftitle={Characterization of the centrally finite Amitsur-Small division rings},
  pdfauthor={Elad Paran}
}

\begin{document}
\maketitle

\begin{abstract}
In 1978, Amitsur and Small asked whether, for every division ring $D$, maximal left ideals of $D[x_1,\ldots,x_n]$ contract to maximal left ideals in smaller polynomial subrings. Chapman and the author showed that the answer is negative in general and called $D$ an \emph{\AS{} ring} when this contraction property always holds. Earlier work showed that Hamilton's real quaternion algebra is \AS{}, that division algebras of degree three are not \AS{}, and that degree-two examples are restricted to a specific quaternionic form; subsequent work excluded cyclic division algebras of odd prime degree. We give a complete resolution in the centrally finite case. If $D$ has finite dimension over its center $F$, then $D$ is an \AS{} ring if and only if either $D=F$, or $F$ is real closed and $D$ is the Hamilton quaternion algebra $(-1,-1)_F$. In every other noncommutative centrally finite case, failure already occurs in two variables: there is a maximal left ideal $M\subseteq D[x,y]$ such that $M\cap D[x]$ is not maximal in $D[x]$.
\end{abstract}



\section{Introduction}

Let $D$ be a division ring, and let $D[x_1,\ldots,x_n]$ denote the ring of polynomials in $n$ central variables over $D$. In their 1978 paper \cite{AmitsurSmall}, Amitsur and Small proved that if $M$ is a maximal left ideal of $D[x_1,\ldots,x_n]$, then the quotient $D[x_1,\ldots,x_n]/M$ is finite-dimensional as a left vector space over $D$. They called this result a \emph{Nullstellensatz over division rings}. A key step in their proof states that
\[
M\cap D[x_1,\ldots,x_k]\ne 0
\]
for every $1\le k<n$. In the commutative case, the contraction is necessarily maximal. Amitsur and Small asked whether the same is true over an arbitrary division ring.

Robson and Small later proved that contractions of semi-maximal left ideals remain semi-maximal \cite{RobsonSmall}. Maximality, however, does not persist in general. In \cite{ChapmanAuthor}, Chapman and the author introduced the following terminology: a division ring $D$ is called an \emph{\AS{} ring} if maximal left ideals in polynomial rings over $D$ contract to maximal left ideals in every smaller polynomial subring. In that work, we constructed examples of non-\AS{} division rings, showed that Hamilton's real quaternion algebra is \AS{}, and obtained restrictions in degrees two and three. More precisely, division algebras of degree three are never \AS{}, while a degree-two division algebra can be \AS{} only if it is a Hamilton quaternion algebra $(-1,-1)_F$ over a Pythagorean field $F$. Subsequently, Chapman, Levin and Zaninelli \cite{ChapmanLevinZaninelli} excluded cyclic division algebras of odd prime degree.

In the present work we characterize all centrally finite Amitsur--Small division rings and thus complete, in the centrally finite case, the line of investigation initiated by Amitsur and Small. We prove:

\begin{theorem}\label{thm:main}
Let $D$ be a division ring of finite dimension over its center $F$. Then the following conditions are equivalent:
\begin{enumerate}
\item $D$ is an \AS{} ring;
\item either $D=F$, or $F$ is real closed and
\[
D\cong (-1,-1)_F.
\]
\end{enumerate}
If $D$ is noncommutative, these conditions are also equivalent to $D$ being left algebraically closed (equivalently, right algebraically closed). Moreover, if $D$ is noncommutative and does not satisfy the second condition, then there exists a maximal left ideal $M$ of $D[x,y]$ such that $M\cap D[x]$ is not maximal in $D[x]$.
\end{theorem}

The proof of the above theorem uses a splitting-field obstruction established in Proposition~\ref{prop:obstruction}: if a finite separable splitting field $K$ of a noncommutative central division algebra $D$ is not algebraically closed, then a counterexample to the \AS{} property can be constructed already in $D[x,y]$. The obstruction is expressed intrinsically in terms of annihilators of vectors in simple modules over split central simple algebras. Consequently, every finite separable splitting field of a noncommutative centrally finite \AS{} division ring must be algebraically closed. We then choose a separable maximal subfield $K\subseteq D$. Since such a field is a splitting field, it must be algebraically closed, and hence every nonconstant polynomial in $F[x]$ has a root in $K\subseteq D$. We then apply a theorem of Baer, which states that if every nonconstant polynomial over the center $F$ of a noncommutative centrally finite division ring $D$ has a root in $D$, then $F$ is real closed and $D$ is Hamilton's quaternion algebra over $F$. This gives the classification directly. The equivalence with left and right algebraic closedness then follows from the Niven--Jacobson theorem. For the positive quaternionic case, the centralizer of each noncentral element is an algebraically closed quadratic field. A recent theorem of Aryapoor concerning the central Nullstellensatz, together with the characterization obtained by Chapman and the author, implies the \AS{} property.

The finite-dimensional hypothesis is essential to the present argument, and no comparable classification is currently available for division rings of infinite dimension over their centers. It remains an open question to characterize the noncommutative centrally infinite \AS{} division rings.

This paper is organized as follows. Section~\ref{sec:prelim} reviews the necessary standard results. Section~\ref{sec:obstruction} establishes the mentioned splitting-field obstruction. Section~\ref{sec:quaternion} proves the positive result for Hamilton quaternion algebras over real closed fields. Section~\ref{sec:classification} proves the classification using algebraically closed maximal subfields and Baer's theorem, and concludes with a concrete example over the field of real constructible numbers.

\section{Preliminaries}\label{sec:prelim}

Throughout this paper, all rings are associative and unital, and all polynomial variables are central. If $D$ is a central division algebra over a field $F$ and $K/F$ is a field extension, we write
\[
D_K=D\otimes_F K.
\]
The \emph{degree} of $D$ is the integer $d$ satisfying $[D:F]=d^2$. A central simple algebra $A$ over a field $K$ is called \emph{split} if
\[
A\cong \operatorname{End}_K(V)
\]
for some finite-dimensional $K$-vector space $V$.

\subsection{The \AS{} property}

\begin{definition}
Let $D$ be a division ring. We say that $D$ is an \emph{\AS{} ring} if, for every $n\ge 2$, every $1\le k<n$, and every maximal left ideal
\[
M\subset D[x_1,\ldots,x_n],
\]
the contraction
\[
M\cap D[x_1,\ldots,x_k]
\]
is a maximal left ideal of $D[x_1,\ldots,x_k]$.
\end{definition}

In the above definition, it is equivalent to require the condition only for $k=n-1$ and then apply it repeatedly. 

\subsection{Standard field-theoretic facts}

We shall use the following standard results, see \cite[Sections~12.7, 13.1, 13.2 and 13.5]{Pierce},  theorem; \cite[Chapter~6, Section~5, Theorem~12]{Ramanathan}, \cite[Chapter~2, Section~6]{Ramanathan}.

\begin{theorem}\label{thm:standard}
Let $D$ be a finite-dimensional central division algebra over a field $F$, and write $d=\deg D$.
\begin{enumerate}[label=(\alph*)]
\item There exists a finite separable extension $K/F$ such that $D_K$ is split.
\item The Artin--Schreier Theorem: If an algebraic closure of $F$ has finite degree over $F$, then $F$ is either algebraically closed or real closed. In the latter case its algebraic closure is $F(\sqrt{-1})$ and has degree two over $F$. 
\item The algebra $D$ contains a maximal subfield $K$ which is separable over $F$. Every such $K$ satisfies $[K:F]=d$ and is a splitting field of $D$.
\end{enumerate}
\end{theorem}

\subsection{One-sided algebraic closedness}

Let $D$ be a division ring and let
\[
p(x)=a_0+a_1x+\cdots+a_mx^m\in D[x].
\]
For $r\in D$, define the right and left values of $p$ at $r$ by
\[
p_r(r)=a_0+a_1r+\cdots+a_mr^m,
\qquad
p_\ell(r)=a_0+ra_1+\cdots+r^m a_m.
\]
The element $r$ is called a \emph{right root} or a \emph{left root} according as $p_r(r)=0$ or $p_\ell(r)=0$. The division ring $D$ is called \emph{right algebraically closed} or \emph{left algebraically closed} if every nonconstant polynomial in $D[x]$ has a right root or a left root, respectively. The factor theorem identifies right roots with right linear factors and left roots with left linear factors; see \cite[Proposition~16.2]{Lam}.

We shall use the following classical classification.

\begin{theorem}[Niven--Jacobson--Baer]\label{thm:NJB}
Let $D$ be a noncommutative division ring of finite dimension over its center $F$. The following conditions are equivalent:
\begin{enumerate}
\item every nonconstant polynomial in $F[x]$ has a root in $D$;
\item $D$ is right algebraically closed;
\item $D$ is left algebraically closed;
\item $F$ is real closed and
\[
D\cong(-1,-1)_F.
\]
\end{enumerate}
\end{theorem}

The equivalence of (ii), (iii), and (iv) is recorded by Niven in \cite[Theorem~1]{Niven}; see also \cite[Theorems~16.14 and 16.15]{Lam}. Niven's theorem is stated under the hypothesis that every equation of the indicated polynomial form with coefficients in $D$ has a solution in $D$. Immediately after the statement, however, Niven observes that Baer's proof of the necessity direction uses only equations whose coefficients lie in the center $F$; see also the proof in \cite[Section~6]{Niven}. After dividing by the leading coefficient, condition~(i) is exactly this central-coefficient hypothesis. Thus Baer's theorem, as presented by Niven, gives (i)$\Rightarrow$(iv) directly. The implications (ii)$\Rightarrow$(i) and (iii)$\Rightarrow$(i) are immediate, while (iv)$\Rightarrow$(ii) and (iv)$\Rightarrow$(iii) are the Niven--Jacobson theorem.

\subsection{The central Nullstellensatz}

Following Aryapoor \cite{Aryapoor}, a division ring $D$ is called \emph{centrally algebraically closed} if the centralizer
\[
\cent_D(a_1,\ldots,a_r)
\]
is right algebraically closed for every finite tuple of pairwise commuting elements $a_1,\ldots,a_r\in D$. Following \cite{ChapmanAuthor}, we call $D$ a \emph{Nullstellensatz ring} if, for every $n\ge 1$, the maximal left ideals of $D[x_1,\ldots,x_n]$ are precisely the ideals
\[
D[x_1,\ldots,x_n](x_1-a_1)+\cdots+D[x_1,\ldots,x_n](x_n-a_n)
\]
with $(a_1,\ldots,a_n)$ pairwise commuting.

\begin{theorem}\label{thm:nullstellensatz}
Let $D$ be a division ring.
\begin{enumerate}[label=(\alph*)]
\item $D$ is a Nullstellensatz ring if and only if $D$ is centrally algebraically closed.
\item If $D$ is right algebraically closed, then
\[
\begin{aligned}
D\text{ is \AS{}}
&\Longleftrightarrow D\text{ is a Nullstellensatz ring}\\
&\Longleftrightarrow D\text{ is centrally algebraically closed}.
\end{aligned}
\]
\end{enumerate}
\end{theorem}

Part (a) is the main result of \cite[Section~4]{Aryapoor}. For right algebraically closed division rings, the equivalence between the \AS{} and Nullstellensatz properties is \cite[Theorem~4.2]{ChapmanAuthor}; Aryapoor also gives the resulting three-way equivalence in \cite[Section~4]{Aryapoor}.

\subsection{Hamilton quaternion algebras}

Let $F$ be a field of characteristic different from $2$, and put
\[
H_F=(-1,-1)_F=F+Fi+Fj+Fk,
\]
where
\[
i^2=j^2=-1,
\qquad ij=-ji=k.
\]
A nonzero element of a quaternion algebra is invertible if and only if its reduced norm is nonzero; see \cite[Lemma~3.3.6]{Voight}. If $F$ is real closed, the reduced norm of $H_F$ is
\[
N(a+bi+cj+dk)=a^2+b^2+c^2+d^2,
\]
which is anisotropic over $F$. Hence $H_F$ is a division algebra.

\section{The splitting-field obstruction}\label{sec:obstruction}

We begin with the following general observation.

\begin{lemma}\label{lem:pullback}
Let $D$ be a central division algebra over $F$, and let
\[
F\subseteq K=F(\theta)\subseteq E=K(\alpha)
\]
be a tower of finite field extensions. Consider the epimorphisms
\[
\Phi:D[x,y]\longrightarrow D_E,
\qquad x\longmapsto\theta,
\qquad y\longmapsto\alpha,
\]
and
\[
\Psi:D[x]\longrightarrow D_K,
\qquad x\longmapsto\theta.
\]
Let $N$ be a left ideal of $D_E$, and put $M=\Phi^{-1}(N)$. Then
\[
M\cap D[x]=\Psi^{-1}(N\cap D_K).
\]
Consequently, if $N$ is maximal in $D_E$ and $N\cap D_K$ is not maximal in $D_K$, then $M$ is maximal in $D[x,y]$ and $M\cap D[x]$ is not maximal in $D[x]$.
\end{lemma}

\begin{proof}
The maps are surjective because $K=F(\theta)$ and $E=K(\alpha)$, and the variables are mapped to central elements. Moreover, the restriction of $\Phi$ to $D[x]$ is the composition of $\Psi$ with the natural inclusion $D_K\subseteq D_E$. Thus, for $f\in D[x]$,
\[
f\in M\Longleftrightarrow \Psi(f)\in N\cap D_K,
\]
which proves the equality.

If $N$ is maximal, then $M$ is maximal because $\Phi$ is surjective. Finally,
\[
D[x]/(M\cap D[x])\cong D_K/(N\cap D_K)
\]
as left $D[x]$-modules, where the $D[x]$-module structure on the right is induced by the surjection $\Psi:D[x]\to D_K$. Hence the $D[x]$-submodules of $D_K/(N\cap D_K)$ are exactly its $D_K$-submodules, and therefore
\[
\begin{aligned}
M\cap D[x]\text{ is maximal in }D[x]
&\Longleftrightarrow D[x]/(M\cap D[x])\text{ is simple}\\
&\hspace{2.4em}\text{as a left }D[x]\text{-module}\\
&\Longleftrightarrow D_K/(N\cap D_K)\text{ is a simple left }D_K\text{-module}\\
&\Longleftrightarrow N\cap D_K\text{ is maximal in }D_K.
\end{aligned}
\]
The last equivalence proves the asserted nonmaximality of $M\cap D[x]$.
\end{proof}

\begin{lemma}\label{lem:annihilator}
Let $A$ be a split central simple algebra of degree $d\geq 2$ over a field $K$, let $V$ be a simple left $A$-module, and let $E=K(\alpha)$ be a proper finite extension. Choose linearly independent elements $u_1,u_2\in V$, put
\[
A_E=A\otimes_K E,
\qquad
V_E=V\otimes_K E,
\qquad
w=u_2\otimes 1-u_1\otimes\alpha,
\]
and let
\[
N=\operatorname{Ann}_{A_E}(w)=\{a\in A_E:aw=0\}.
\]
Then $N$ is a maximal left ideal of $A_E$, and
\[
N\cap A=\operatorname{Ann}_A(u_1,u_2)
=\{a\in A:au_1=au_2=0\}.
\]
Moreover,
\[
A/(N\cap A)\cong V\oplus V
\]
as left $A$-modules. In particular, $N\cap A$ is not maximal and the quotient has composition length two. Finally, $N=A_E\pi$ for an idempotent $\pi\in A_E$.
\end{lemma}

\begin{proof}
The left $A$-module structure on $V$ defines a $K$-algebra homomorphism
\[
A\longrightarrow\operatorname{End}_K(V),
\qquad
 a\longmapsto(v\longmapsto av).
\]
Since $A$ is split and $V$ is simple, the Wedderburn--Artin theorem identifies this homomorphism with an isomorphism
\[
A\cong\operatorname{End}_K(V).
\]
After extending scalars from $K$ to $E$, the induced left $A_E$-module structure on $V_E$ similarly gives
\[
A_E\cong\operatorname{End}_E(V_E),
\]
and $V_E$ is a simple left $A_E$-module. These standard facts about split simple algebras and their simple modules are contained in \cite[Theorem~3.3]{Lam}; see also \cite[Chapter~1, Sections~2--3]{Lam} for the Wedderburn--Artin theory.

The vector $w$ is nonzero, since the linearly independent vectors $u_1,u_2$ remain linearly independent after extending scalars. Give $A_E$ its left regular module structure and $V_E$ the left $A_E$-module structure just described. Then the evaluation map
\[
A_E\longrightarrow V_E,
\qquad a\longmapsto aw,
\]
is a nonzero homomorphism of left $A_E$-modules. Its image is a nonzero submodule of the simple module $V_E$, and is therefore all of $V_E$. Its kernel is $N$, so $N$ is maximal.

For $a\in A$, one has
\[
aw=au_2\otimes 1-au_1\otimes\alpha.
\]
Since $1$ and $\alpha$ are linearly independent over $K$, this vanishes if and only if $au_1=au_2=0$. Hence
\[
N\cap A=\operatorname{Ann}_A(u_1,u_2).
\]
The map
\[
A\longrightarrow V\oplus V,
\qquad a\longmapsto(au_1,au_2),
\]
is surjective: prescribed images of the linearly independent vectors $u_1,u_2$ extend to a $K$-linear endomorphism of $V$, and the isomorphism $A\cong\operatorname{End}_K(V)$ shows that every such endomorphism is induced by an element of $A$. Its kernel is $N\cap A$, which gives the asserted isomorphism. Since $V$ is simple, the quotient has composition length two.

Choose a decomposition $V_E=Ew\oplus U$, and let $\pi$ be the projection onto $U$ along $Ew$. Then $\pi^2=\pi$ and $\pi(w)=0$. An endomorphism $a$ of $V_E$ annihilates $w$ if and only if $a=a\pi$. Therefore
\[
N=A_E\pi.
\]
\end{proof}

\begin{proposition}[Splitting-field obstruction]\label{prop:obstruction}
Let $D$ be a central division algebra of degree $d>1$ over $F$, and let $K/F$ be a finite separable splitting field of $D$. If $K$ is not algebraically closed, then $D$ is not an \AS{} ring. Moreover, a counterexample may be chosen in $D[x,y]$.
\end{proposition}

\begin{proof}
By the primitive element theorem, we may write $K=F(\theta)$. Since $K$ is not algebraically closed, there is a proper finite extension $E=K(\alpha)$. Put $A=D_K$, and choose a simple left $A$-module $V$. Since $A$ is split of degree $d>1$, the $K$-space $V$ has dimension $d$, so it contains linearly independent elements $u_1,u_2$.

Apply Lemma~\ref{lem:annihilator} to $A$, $V$ and $E/K$. Under the natural identification $A_E\cong D_E$, it gives a maximal left ideal $N$ of $D_E$ such that $N\cap D_K$ is not maximal in $D_K$. Lemma~\ref{lem:pullback} now gives a maximal left ideal $M$ of $D[x,y]$ whose contraction to $D[x]$ is not maximal.
\end{proof}

\begin{remark}\label{rem:generators}
The construction may also be described by generators. Let $p(X)\in F[X]$ be the minimal polynomial of $\theta$. Let $q(Y)\in K[Y]$ be the minimal polynomial of $\alpha$, and choose a polynomial $\widetilde q(X,Y)\in F[X,Y]$ whose image in
\[
F[X,Y]/(p(X))\cong K[Y]
\]
is $q(Y)$. The homomorphism $F[X,Y]\to E$ given by $X\mapsto\theta$ and $Y\mapsto\alpha$ has kernel $(p(X),\widetilde q(X,Y))$: after quotienting by $p(X)$, it becomes the map $K[Y]\to E$ with kernel $(q(Y))$. Tensoring the resulting exact sequence over $F$ with the $F$-vector space $D$, and using $D[x,y]\cong D\otimes_F F[X,Y]$ and $D_E=D\otimes_F E$, gives
\[
\ker\Phi=D[x,y]p(x)+D[x,y]\widetilde q(x,y).
\]

The ideal $N$ in the proof of Proposition~\ref{prop:obstruction} is principal by the last assertion of Lemma~\ref{lem:annihilator}. Write
\[
N=D_E\pi
\]
for the idempotent $\pi$ supplied there, and choose $g\in D[x,y]$ with $\Phi(g)=\pi$. Since $\Phi$ is surjective,
\[
\Phi^{-1}(N)=D[x,y]g+\ker\Phi.
\]
Consequently, the maximal left ideal constructed above is
\begin{equation}\label{eq:threegen}
M=D[x,y]g+D[x,y]p(x)+D[x,y]\widetilde q(x,y).
\end{equation}
Thus the counterexample can always be taken to be a three-generated left ideal.
\end{remark}

\begin{corollary}\label{cor:splitting-closed}
Let $D$ be a noncommutative centrally finite \AS{} division ring. Then every finite separable splitting field of $D$ is algebraically closed. In particular, the center of $D$ has a finite algebraic closure.
\end{corollary}

\begin{proof}
The first claim follows from Proposition~\ref{prop:obstruction}. A finite separable splitting field exists by Theorem~\ref{thm:standard}(a), and this field is algebraic over the center. If it is algebraically closed, it is an algebraic closure of the center.
\end{proof}

\section{Hamilton quaternion algebras over real closed fields}\label{sec:quaternion}

The central quaternionic Nullstellensatz for Hamilton's real quaternion algebra was proved by Alon and the author in \cite{AlonAuthor}. Here the positive direction follows from Aryapoor's general characterization. We record two centralizer lemmas needed for the application.

\begin{lemma}\label{lem:centralizer-one}
Let $F$ be real closed, let $D=(-1,-1)_F$, and let $a\in D\setminus F$. Then
\[
\cent_D(a)=F(a),
\]
and this field is algebraically closed.
\end{lemma}

\begin{proof}
Quaternionic conjugation gives
\[
a^2-(a+\overline a)a+a\overline a=0,
\]
whose coefficients lie in $F$. Since $a\notin F$, its minimal polynomial over $F$ has degree two, and hence $[F(a):F]=2$. By the double-centralizer theorem,
\[
\dim_F\cent_D(a)=\frac{\dim_F D}{[F(a):F]}=2;
\]
see \cite[Section~12.7]{Pierce}. Since $F(a)\subseteq\cent_D(a)$ and both have dimension two over $F$, they are equal. Finally, a proper finite extension of a real closed field is its quadratic algebraic closure, so Theorem~\ref{thm:standard}(b) shows that $F(a)$ is algebraically closed.
\end{proof}

\begin{lemma}\label{lem:centrally-closed}
Let $F$ be real closed and $D=(-1,-1)_F$. Then $D$ is centrally algebraically closed.
\end{lemma}

\begin{proof}
Let $a_1,\ldots,a_r$ be pairwise commuting elements of $D$. If all of them lie in $F$, then their common centralizer is $D$, which is right algebraically closed by Theorem~\ref{thm:NJB}. Otherwise choose $a_i\notin F$. Lemma~\ref{lem:centralizer-one} gives
\[
\cent_D(a_i)=F(a_i).
\]
Every $a_j$ belongs to this commutative field because it commutes with $a_i$. Consequently, every element of $F(a_i)$ commutes with all the $a_j$, and hence
\[
\cent_D(a_1,\ldots,a_r)=F(a_i).
\]
This field is algebraically closed by Lemma~\ref{lem:centralizer-one}. Thus every common centralizer of a commuting tuple is right algebraically closed.
\end{proof}

\begin{proposition}\label{prop:positive}
Let $F$ be real closed and let
\[
D=(-1,-1)_F.
\]
Then $D$ is an \AS{} ring.
\end{proposition}

\begin{proof}
By Theorem~\ref{thm:NJB}, the division ring $D$ is right algebraically closed. By Lemma~\ref{lem:centrally-closed}, it is centrally algebraically closed. Theorem~\ref{thm:nullstellensatz}(b), which combines \cite[Theorem~4.2]{ChapmanAuthor} with Aryapoor's central Nullstellensatz, now implies that $D$ is \AS{}.
\end{proof}

\begin{remark}
For the application of \cite[Theorem~4.2]{ChapmanAuthor}, one-sided algebraic closedness must be supplemented by the Nullstellensatz property. In the quaternionic case this additional condition follows from the centralizer computation above and \cite[Section~4]{Aryapoor}.
\end{remark}

\section{Classification}\label{sec:classification}

\begin{proof}[Proof of Theorem~\ref{thm:main}]
If $D=F$, then $D$ is \AS{} by the commutative case. If $F$ is real closed and $D\cong(-1,-1)_F$, then $D$ is \AS{} by Proposition~\ref{prop:positive}.

Conversely, suppose that $D$ is \AS{}. If $D$ is commutative, then $D=F$. Suppose that $D$ is noncommutative, and choose a maximal subfield $K\subseteq D$ which is separable over $F$. By Theorem~\ref{thm:standard}(c), the field $K$ is a finite separable splitting field of $D$. Corollary~\ref{cor:splitting-closed} therefore shows that $K$ is algebraically closed. Every nonconstant polynomial in $F[x]$ consequently has a root in $K\subseteq D$. Theorem~\ref{thm:NJB}, in Baer's central-root form, now gives
\[
F\text{ real closed},
\qquad D\cong(-1,-1)_F.
\]
The equivalence with left and right algebraic closedness follows from the remaining parts of Theorem~\ref{thm:NJB}.

It remains to prove the final assertion. Suppose that $D$ is noncommutative and is not of the displayed form. Choose a maximal subfield $K\subseteq D$ which is separable over $F$. If $K$ were algebraically closed, every nonconstant polynomial in $F[x]$ would have a root in $K\subseteq D$, and Theorem~\ref{thm:NJB} would imply that $F$ is real closed and $D\cong(-1,-1)_F$, a contradiction. Thus $K$ is not algebraically closed. Since $K$ is a finite separable splitting field by Theorem~\ref{thm:standard}(c), Proposition~\ref{prop:obstruction} supplies a maximal left ideal $M\subseteq D[x,y]$ such that $M\cap D[x]$ is not maximal in $D[x]$.
\end{proof}

We conclude with a concrete specialization in which the maximal left ideal has particularly simple generators.

\begin{example}[The field of real constructible numbers]\label{ex:constructible}
Let $F=\mathbb{Q}_{\mathrm{constr}}$ be the field of real constructible numbers, and put
\[
D=(-1,-1)_F.
\]
Recall that every element of $F$ lies in the top field of a finite tower of quadratic extensions of $\mathbb{Q}$. Since $F\subseteq\mathbb{R}$, it is formally real. Moreover, if $a,b\in F$, then $\sqrt{a^2+b^2}$ is constructible, so $F$ is Pythagorean. In particular, the reduced norm of $D$ is anisotropic and $D$ is a division algebra. It is not real closed, since $\sqrt[3]{2}\notin F$. Hence $D$ is not \AS{} by Theorem~\ref{thm:main}.

We now write down an explicit counterexample. Let
\[
K=F(\varepsilon),
\qquad \varepsilon^2=-1,
\qquad \alpha=\sqrt[3]{2},
\qquad E=K(\alpha).
\]
The element $\alpha$ does not belong to $K$: every element of $K$ is contained in a finite extension of $\mathbb{Q}$ of degree a power of two, whereas $T^3-2$ is irreducible over $\mathbb{Q}$ by Eisenstein's criterion and hence $[\mathbb{Q}(\alpha):\mathbb{Q}]=3$. The same argument applies to every root of $T^3-2$, so the polynomial has no root in $K$. Since it has degree three, it is irreducible over $K$.

Since $F$ is formally real, $-1$ is not a square in $F$. Thus $[F(i):F]=2=\deg D$, so $F(i)$ is a maximal subfield of $D$ and hence splits $D$ by Theorem~\ref{thm:standard}(c). Let
\[
\sigma:F(i)\longrightarrow K
\]
be the $F$-isomorphism determined by $\sigma(i)=\varepsilon$. The induced isomorphism
\[
\operatorname{id}_D\otimes\sigma:
D\otimes_F F(i)\longrightarrow D\otimes_F K
\]
shows that $K$ also splits $D$. Put $B=D_K$ and
\[
e=\frac{1-\varepsilon i}{2}\in B.
\]
Since $(\varepsilon i)^2=1$, the element $e$ is an idempotent. If $z=\varepsilon i$, then $zj=-jz$. Using the $K$-basis $1,z,j,zj$ of $B$, one obtains
\[
Be=Ke\oplus Kje,
\qquad eBe=Ke.
\]
Since $eBe$ is a division ring, $e$ is primitive; equivalently, $V=Be$ is a minimal nonzero left ideal of the simple Artinian ring $B$, and hence is a simple left $B$-module. See \cite[Chapter~1, Sections~2--3]{Lam}. It has $K$-basis
\[
u_1=e,
\qquad u_2=-je.
\]
The annihilator $\operatorname{Ann}_B(V)$ is a two-sided ideal of the simple ring $B$. Since $V\neq0$, this annihilator is proper and therefore zero. Thus $V$ is faithful.

In $V_E=V\otimes_K E$, set
\[
w=u_2\otimes1-u_1\otimes\alpha=-(j+\alpha)e
\]
and
\[
N=\operatorname{Ann}_{D_E}(w).
\]
By Lemma~\ref{lem:annihilator}, the left ideal $N$ is maximal. Since $u_1,u_2$ form a basis of the faithful $B$-module $V$, Lemma~\ref{lem:annihilator} gives
\[
N\cap B=0.
\]

Let
\[
h=(1-\varepsilon i)(1+\alpha j)=2e(1+\alpha j)\in D_E.
\]
The relation $eje=0$ follows from $zj=-jz$, and hence
\[
hw
= -2e(1+\alpha j)(j+\alpha)e
= -2(1+\alpha^2)eje
=0.
\]
Therefore $D_Eh\subseteq N$. The element $1+\alpha j$ is invertible, with inverse
\[
\frac{1-\alpha j}{1+\alpha^2}.
\]
Here $1+\alpha^2\ne 0$, since $T^3-2$ is the minimal polynomial of $\alpha$ over $K$. Right multiplication by $1+\alpha j$ therefore identifies $D_Ee$ with $D_Eh$, so $D_Eh$ has dimension two over $E$. The evaluation sequence in the proof of Lemma~\ref{lem:annihilator} gives $D_E/N\cong V_E$, and hence $N$ also has dimension two over $E$. Consequently,
\[
N=D_Eh.
\]

Let $R=D[s,t]$, and consider the left ideal
\begin{equation}\label{eq:explicit}
M=R(s^2+1)+R(t^3-2)+R(1-si)(1+tj).
\end{equation}
Under the epimorphism
\[
\Phi:R\longrightarrow D_E,
\qquad s\longmapsto\varepsilon,
\qquad t\longmapsto\alpha,
\]
the first two generators of \eqref{eq:explicit} vanish and the third maps to $h$. Remark~\ref{rem:generators} gives
\[
\ker\Phi=R(s^2+1)+R(t^3-2).
\]
Since $N=D_Eh$, it follows that $M=\Phi^{-1}(N)$, and hence $M$ is maximal.

The image of $D[s]$ is $B=D_K$. Since $N\cap B=0$, Lemma~\ref{lem:pullback} gives
\[
M\cap D[s]
=
\ker\bigl(D[s]\longrightarrow D_K\bigr)
=
D[s](s^2+1).
\]
Finally,
\[
D[s]/D[s](s^2+1)\cong D_K=B.
\]
The zero left ideal of $B$ is not maximal, since $Be$ is a nonzero proper left ideal. Thus the displayed contraction is not maximal.
\end{example}

\section*{Acknowledgements}
\noindent\textbf{Declaration of AI use.} This paper was written with significant AI assistance. The model used was ChatGPT 5.5 pro. The starting point was however a human proof: The example of the field of constructible numbers (Example \ref{ex:constructible}) was found by the author in order to answer a question that was raised by Chapman and the author in \cite{ChapmanAuthor}. This Example was ``fed'' to the AI model, which led to further and further generalizations via a series of prompts, ending in the main theorem of this paper. The initial proof however was more complicated than the one given here: The AI model seemed ``unaware'' of the theorem of Baer characterizing centrally finite algebraically closed division rings as quaternion algebras over a real-closed field. By pointing the model to this theorem, a simpler and more direct proof was achieved -- this is the proof given in this manuscript. The end result is a human-AI collaborative work.

\end{document}